\documentclass[12pt]{amsart}

\usepackage{fontenc}
\usepackage{amssymb,amsmath,amscd,latexsym,amsfonts,amstext,amsbsy,amsthm}
\usepackage{euscript}
\usepackage{enumerate}

\usepackage[centertags]{amsmath}
\usepackage{amsfonts}
\usepackage{amssymb}
\usepackage{amsthm}
\usepackage{newlfont}

 \newtheorem{thm}{Theorem}[section]
 \newtheorem{cor}[thm]{Corollary}
 \newtheorem{lem}[thm]{Lemma}
 \newtheorem{prop}[thm]{Proposition}
 
 \theoremstyle{definition}
 \newtheorem{defn}[thm]{Definition}

 \numberwithin{equation}{subsection}
  \newtheorem{claim}[thm]{Claim}

 \newcommand{\N}{\mathbb{N}}

 \newcommand{\norm}[1]{\left\Vert#1\right\Vert}

\newcommand{\bp}{\begin{problem}}

\newcommand{\ep}{\end{problem}}

\newcommand{\ben}{\begin{enumerate}}
\newcommand{\een}{\end{enumerate}}

\newcommand{\valu}[1]{\left| #1 \right|}

\begin{document}

\title{A Coanalytic Rank on Super-Ergodic Operators}
\author{Mohammed Yahdi}

\address{Department of  Mathematics and
Computer Science\\
Ursinus College\\
Collegeville,  PA 19426, USA}

\email{myahdi@ursinus.edu}



\maketitle

\begin{abstract}
Techniques from Descriptive Set Theory are applied in order to
study the Topological Complexity of families of operators
naturally connected to ergodic operators in infinite dimensional
Banach Spaces. The families of ergodic, uniform-ergodic,
 Cesaro-bounded and power-bounded operators are shown to be Borel sets, while the family of super-ergodic operators is shown
to be either coanalytic or Borel according to specific structures of the
Space. Moreover, trees and coanalytic ranks are introduced
to characterize super-ergodic operators as well as spaces where
the above classes of operators do not coincide.

\end{abstract}


\section{Introduction}

Let  $T$ be a bounded operator on an infinite dimensional Banach
space $X$, and let {$
A_n = \frac{1}{n}
\sum_{k=0}^{n-1} T^k$} be  the \textit{$n^{th}$-Cesaro-mean} of $T$.
Consider the following definitions:
\begin{itemize}

  \item $T$ is ergodic if the sequence $\{A_n\}_{x\geq1}$ 
  converges in the space of operators $L(X)$ equipped with the
  strong operator topology $S_{op}$
  \item $T$ is uniformly ergodic if the sequence $\{A_n\}_{x\geq1}$
  converges in $L(X)$ equipped with its natural norm.
  \item $T$ is weakly ergodic if for any $x\in X$, the sequence
  $\{A_n(x)\}_{x\geq1}$ weakly converges in $X$.
  \item $T$ is Cesaro-bounded if the norms of   $\{A_n\}_n$ are uniformly
  bounded.
  \item $T$ is power-bounded if the  norms of $\{T^n\}_n$ are uniformly
  bounded.
\end{itemize}

The definition of a super-ergodic operator is introduced in
\cite{yahdi3} as the super property associated with ergodic
operators. Let $\ell^{\infty}(X)$ be the Banach space of bounded
sequences  in $X$ and let
  $\mathcal{C_{\mathcal{U}}}(X)$ be the subspace  of sequences $\{x_n\}_n$ such that:
  {$
\lim_{\mathcal{U}} \norm{x_n}=0$}.
$T$ is super-ergodic if for any
 ultrafilter $\mathcal{U}$ on $\mathbb{N}$, the ultrapower operator
$T_{\mathcal{U}}$ is
 ergodic on  the ultraproduct
 {$\displaystyle X_\mathcal{U}:=  \ell^{\infty}(X) / \mathcal{C_{\mathcal{U}}}(X)  $},
where
 \[
T_{\mathcal{U}}(\overline{x})  := \{Tx_n\}_n +
\mathcal{C_{\mathcal{U}}}(X) \quad \text{ for any } \overline{x}:=
\{x_n\}_n + \mathcal{C_{\mathcal{U}}}(X).
 \]

In \cite{yahdi3} it was shown  that uniformly ergodic implies
super-ergodic, which in turn implies ergodic. Examples were given
to show that these implications are strict. An interesting question is to
 determine the structures of the Banach space that place
strong
  limits on these relationships that
  might be possible instead of just particular examples.
Translations of certain mathematical concepts to families of sets, then
examining their positions in the descriptive set hierarchy have been
proven to be a productive approach (see \cite{yahdi1} and
\cite{yahdi2}). For example,  \cite{yahdi2} investigates the
   ``family of  superstable operators'' where its topological complexity is shown to be connected
    to the structure of the space; such as having some kind of unconditional
  basis or being hereditarily indecomposable. In particular, this produced families of spaces where stable, superstable and uniformly-stable operators are either equivalent or strictly separated.
    The important results in  \cite{Rose}, \cite{gowers2} and \cite{maurey1}
    show the power of some applications of descriptive set theory.

In this work, some techniques from \cite{yahdi2} are applied to sets related to
classes of ergodic operators. For the particular class of super-ergodic
operators,  a characterization in terms of
\textit{``trees''} is introduced and more descriptive set theory tools
are applied. It is shown that the height of these trees is a
coanalytic rank, thus proving that the set of
super-ergodic operators is coanalytic for the strong operator
topology.  All the other families of operators are shown to be
Borel. In particular, this gives a general class of Banach spaces
for which super-ergodic is strictly stronger than ergodic and strictly weaker
than uniformly ergodic.
Moreover,  examples are given to show the
existence of such spaces.

Throughout this work, $X$ denotes an infinite dimensional  Banach
space with a norm $\norm{.}$.

\section{Application of descriptive set theory}

In a Polish space $P$,
  we consider the natural  hierarchy {$ \Pi^0_{\xi} $}
and $\Sigma^0_{\xi}$ covering all Borel subsets starting from the
open and closed sets for $\xi=1$  to more complex ones  defined
by induction on the countable ordinals $\xi$; where a
$\Sigma^0_{\xi}$  is a countable union of $\Pi^0_{\xi-1}
  $ sets, a $\Pi^0_{\xi}$ is a  countable intersection of $\Sigma^0_{\xi-1} $ sets, and for
 $\alpha$ a limit ordinal:
 \[
 \Sigma^0_{\alpha} = \bigcup_{\xi<\alpha} \Pi^0_{\xi}
 \quad { and } \quad
 \Pi^0_{\alpha} = \bigcap_{\xi<\alpha} \Sigma^0_{\xi}
 \]

Borel sets are not the only subsets in a Polish space that are
constructible  from open and closed sets  (see \cite{godefroy},
\cite{christensen} or \cite{Kechris}):
\begin{itemize}
  \item $A \subset P$ is analytic if it is a continuous image of a Polish
  space.

  \item $A \subset P$ is coanalytic  if $P\setminus A$ is analytic.

\end{itemize}

 The space
$L(X)$ of bounded operators equipped with the strong operator
topology $S_{op}$
is not a Polish space since it is not a Baire space. However, it is
also possible to work in a standard Borel space; i.e., a space
Borel isomorphic  to a Borel set of a Polish space.

\begin{prop}
For any separable  Banach space $X$,  $L(X)$ of bounded operators is a standard Borel space when equipped with the
$\sigma$-algebra $\sigma(S_{op}\big{)}$ generated by the strong operator topology.
\end{prop}

\begin{proof}
Let $X$ be a separable Banach space. It is well-known that with the strong operator topology the spaces $L_n(X)$ of operators of norm $\leq n $ is a Polish space in which $T  \longmapsto T^n$ is clearly continuous (see, e.g., page 14 in \cite{Kechris} or lemma 3. in  \cite{yahdi2}). So $L(X)= \bigcup_n L_{n+1}(X)\setminus L_n(X)$ is clearly standard Borel when equipped with the $\sigma$-algebra generated by the strong operator topology.
\end{proof}

The following lemma is a modified  result from \cite{Al-Bo}.

\begin{lem} \label{lemme:caract de E}
Let $T$  be a  Cesaro-bounded operator on a  Banach space $X$, and
let $\{A_n\}_{n \in \mathbb{N}}$ be the corresponding sequence of
 Cesaro-means. Then $T$ is ergodic if and only if the sequence
  $\{A_nx\}_{n}$
 converges in norm  for all $x$ in a  dense  (in norm) subset of $X$.
\end{lem}

\begin{proof}
Put
 {$\displaystyle E:=\{ x \in X: \{A_nx\}_{n \in \mathbb{N}} \hbox{ norm-converges in
} X\}.$}
Suppose that $E$ is dense in $X$ and that {$\displaystyle  M:=
\sup_{n \in \mathbb{N}} \norm{A_n} < \infty $}. \\
Let  $y \in X$
be fixed.
 For $
\varepsilon
> 0 $, consider  $ x \in X $ such that ~$ \norm{x-y} <
\varepsilon $. Then, there exits  $n_0 \in \mathbb{N}$ such that
{$\displaystyle ~ \norm{A_nx - A_mx} < \varepsilon, ~ \forall n, m
\geq n_0 . $}  Thus, for all   $n, m \geq n_0 $ ~ we have
\begin{align*}
\norm{A_ny - A_my} &\leq \norm{A_n(y-x)} + \norm{A_nx - A_mx} + \norm{A_m(x - y)} \\
                                                                            &\leq (2M+1) \varepsilon .
\end{align*}
Therefore the sequence $\{A_ny\}_{n \in \mathbb{N}}$ is  Cauchy,
and hence
 $y \in E$.

\end{proof}

Denote by  $\mathcal{E}(X)$, $\mathcal{SE}(X)$, $\mathcal{UE}(X)$,
$L_{cb}(X)$ and $L_{pb}(X)$,
 respectively the subsets of  $L(X)$ of
  ergodic,
super-ergodic, uniformly  ergodic, Cesaro-bounded and power bounded
operators on $X$.

\begin{prop} \label{prop:Cesaro-borne borelien}
For any separable Banach space $X$, the sets  $L_{pb}(X)$ of
power-bounded operators, $L_{cb}(X)$ of Cesaro-bounded operators, $
\mathcal{E}(X) $ of ergodic operators and  $\mathcal{UE}(X)$ of
uniformly ergodic operators
 are all Borel in $\big{(}L(X),
\sigma(S_{op})\big{)}$.
\end{prop}

\begin{proof}

Let $\{x_n\}_{n \in \mathbb{N}}$ be a  dense sequence in the unit
closed ball $B_X$ of $X$. It is not difficult to show that
\[
L_{pb}(X) =
\bigcup_{k \in \mathbb{N}} \bigcap_{m \in \mathbb{N}}
  \bigcap_{n \in \mathbb{N}}
   \big{\{} T \in L(X); ~ \norm{T^m x_n} \leq k \big{\}} .
   \]
It follows from the continuity of $T  \longmapsto T^n$ that $L_{pb}(X)$
is  $F_\sigma$ in  $\big{(}L(X), \sigma(S_{op})\big{)}$.
Similarly for the set $L_{cb}(X)$ using the  $S_{op}$-continuity
on bounded subsets of $L(X)$  of the maps {$
 R   \longmapsto   \frac{1}{n} \sum_{i=0}^{n-1} R^i
$}.

 Let
$T$ be a bounded operator on $X$. By lemma \ref{lemme:caract de E},
$T$ is ergodic if and only if $T$ is   Cesaro-bounded and
 the sequence  {$
 \{ (\frac{1}{n} \sum_{i=0}^{n-1}
T^i)x_k\}_{n} $} is norm-Cauchy for all $k \in \mathbb{N}$; i.e.
\[
\forall  \varepsilon >0, \exists N \in \mathbb{N}: \forall n,m
\geq N \quad
                        \norm{ \frac{1}{n} \sum_{i=0}^{n-1} T^ix_k -
\frac{1}{m} \sum_{i=0}^{m-1} T^ix_k} < \varepsilon.
\]
In other terms,
\[ T \in \bigcap_{k \in  \mathbb{N}} \bigcap_{p \in
\mathbb{N}} \bigcup_{N \in  \mathbb{N}} \bigcap_{n,m \geq N}
         \Big{\{} R \in L(X): \norm{ \frac{1}{n} \sum_{i=0}^{n-1} R^ix_k -
\frac{1}{m} \sum_{i=0}^{m-1} R^ix_k} < \frac{1}{p} \Big{\}}.
\]
This  proves the  result for
 $\mathcal{E}(X)$  using
  the continuity of $T  \longmapsto T^n$.

The same arguments prove the result for  $\mathcal{UE}(X)$ since $T$ is
uniformly ergodic if and only if $T$ is Cesaro-bounded and the sequence
$\{A_n\}_{n\in \N}$ of its Cesaro-means is Cauchy for the norm of
$L(X)$.
\end{proof}

\section{Set of super-ergodic operators}

 First, an entropy-tree will be defined to help
 characterize the super-ergodicity of an operator.
For an ultrafilter $\mathcal{U}$  on $\N$, denote by
$A^\mathcal{U}_n$  the $n^{th}$ Cesaro-mean of the ultrapower
$T_\mathcal{U}$:
\[
A^\mathcal{U}_n=\frac{1}{n} \sum_{k=0}^{n-1} T^k_\mathcal{U}
 \]
By definition, an operator $T$  is not super-ergodic if and only if there exist an ultrafilter
 $\mathcal{U}$ and
$\bar{x} \in  X_\mathcal{U}$ such that $ \{A^\mathcal{U}_n
\bar{x}\}_{n}$ does not converge,  i.e.,
\[
\exists \bar{x} \in B_{X_\mathcal{U}}, \exists \varepsilon>0,
\exists J = \{j_p\}_{p} \in \N^{\uparrow\N}:  \forall p \in \N,
\norm{A^\mathcal{U}_{j_p} \bar{x} - A^\mathcal{U}_{j_{p+1}}
\bar{x}}_{X_\mathcal{U}}> \varepsilon
\]
where $\N^{\uparrow\N}$ is the set of infinite and strictly
increasing sequences of $\N$.

\noindent
 Let $(x_n)_{n \in \N} \in \bar{x}$ be chosen in the unit ball $B_X$. The condition
\[
\forall p \in \N, \quad \norm{A^\mathcal{U}_{j_p} \bar{x} -
A^\mathcal{U}_{j_{p+1}} \bar{x}}_{X_\mathcal{U}}> \varepsilon
\]
is then equivalent to
\[
\forall p \in \N, ~\exists E_p \in \mathcal{U}:~Ê~Ê \forall n \in
E_p,\quad \norm{A_{j_p} x_n - A_{j_{p+1}} x_n}> \varepsilon,
\]
or again, by using  {$\displaystyle E_m=\bigcap_{p\leq m}E_p$},
\[
\forall m \in \N, ~\exists E_m \in \mathcal{U}:~Ê~Ê \forall n \in
E_m,\quad \norm{A_{j_p} x_n - A_{j_{p+1}} x_n}> \varepsilon  ~
\hbox{ } \forall p \leq m.
\]
This implies in particular that
\[
\forall m \in \N, ~\exists x_m \in B_X:\quad \norm{A_{j_p} x_m -
A_{j_{p+1}} x_m}> \varepsilon  ~ \hbox{ } \forall p \leq m.
\]
Therefore, this proves that if $T$ is not super-ergodic,
then
 $T$ satisfies the following condition,  noted $\mathcal{({NSE})}$;
\[
\exists \varepsilon>0,~ \exists J=\{j_p\}_{p \in \N}
\in\N^{\uparrow\N}, ~~ \forall m \in \N, ~\exists x_m \in B_X:\]
\[
\norm{A_{j_p} x_m - A_{j_{p+1}} x_m}> \varepsilon  ~ \hbox{ }
\forall p \leq m.
\]

\begin{lem} \label{lemme:non superergo}
 $T$
is not super-ergodic if and only if  $T$ satisfies
$\mathcal{(NSE)}$.
\end{lem}

\begin{proof}
One direction was proved above. Suppose now that $T$ satisfies
$\mathcal{(NSE)}$.
Let $\mathcal{U}$ be the ultrafilter on $\N$ that contains all sets
 $E_n:=\{n,n+1,\dots\}$. Put  $\bar{x}=(x_m)_{m \in
\N} + C_\mathcal{U}(X)  ~ \in X_\mathcal{U}$.
The condition $\mathcal{(NSE)}$ implies that
\[
\forall m \in  E_p, \quad  \norm{A_{j_p} x_m - A_{j_{p+1}}x_m}>
\varepsilon.
\]
So, {$\displaystyle  \mathcal{U}-\lim_m \norm{A_{j_p} x_m -
A_{j_{p+1}} x_m}> \frac{\varepsilon}{2} $},  and thus
 {$\displaystyle
\norm{A_{j_p} \bar{x} - A_{j_{p+1}}\bar{x}}_\mathcal{U}>
 \varepsilon.$}\\
Since this is true for any positive integer  $p$, it follows that
the sequence
 $\big{\{}A^\mathcal{U}_n(\bar{x}) \big{\}}_{n \in \N}$
is not  Cauchy and hence  $T$ is not  super-ergodic.

\end{proof}

With lemma \ref{lemme:non superergo},  the super-ergodicity can be
described in terms of trees using the following notations:
\begin{itemize}
  \item $\N^{\uparrow<\N}$ denotes the set of finite and strictly increasing
sequences in $ \N$ as well as the empty sequence.
  \item For $s \in
\N^{\uparrow<\N}$, $s_p$ denotes the $p^{th}$ element of $s$ and
$\valu{s}$ denotes the length of $s$.
  \item For $s , s' \in
\N^{\uparrow<\N},  s \prec s'$ means that
$ \valu{s} < \valu{s'}$ and have the same first $\valu{s}$
elements.
\end{itemize}

\begin{defn} \label{def-prop:arbre ergo}
Let $X$ be a  Banach space and $T \in L(X)$ with  $A_n $ its
$n^{th}$ Cesaro-mean. For all $\varepsilon>0$,
$\mathcal{A}_e(T,\varepsilon) $ is the tree on $\N$  defined by  the set
of all elements $s \in \N^{\uparrow<\N}$ such that
\[
 \valu{s} \leq 1 \hbox { or }  ~ \exists ~ x \in B_X \hbox{ such that }
 \forall 1 \leq p <\valu{s},
 ~ \norm{A_{s_p}x - A_{s_{p+1}}x}>
\varepsilon.
\]
\end{defn}

 It follows from Lemma \ref{lemme:non
superergo} that $T$ is not super-ergodic if and only if
\[
\exists \varepsilon >0 \hbox{ and } \exists J \in \N^{\uparrow\N}
~ \hbox{ such that } \quad \forall s\prec J,
 ~ s \in \mathcal{A}(T,\varepsilon).
\]
In other terms, $T$ is not super-ergodic if and only if the tree
$\mathcal{A}_e(T,\varepsilon)$ is not well founded for certain
$\varepsilon>0$; i.e., with  infinite branches.

\begin{thm} \label{prop:indice arbre ergo}
Let $X$ be a Banach space and $T \in L(X)$. The following
assertions are equivalent:
\begin{enumerate}[(a)]
\item $T$  is super-ergodic.

\item For all $\varepsilon >0 $,\quad the tree
$\mathcal{A}_e(T,\varepsilon)$ is well founded.

\item ${\displaystyle \eta_e (T) := \sup_{\varepsilon >0}
h{(}\mathcal{A}_e(T,\varepsilon)\big{)} <\omega_1}$, where $h$
gives the height of a tree.
\end{enumerate}
\end{thm}

\begin{proof}
 The equivalence between $(a)$ and $(b)$ was  shown earlier. The equivalence between
 $(b)$ and $(c)$ is obvious because $(b)$ is equivalent to

\begin{center}
$\forall n  \in \N, \quad \mathcal{A}_e(T,\frac{1}{n}) $ is well
founded.
\end{center}

\end{proof}

The index $\eta_e$ defined in theorem \ref{prop:indice arbre ergo} extends  to all $T \in L(X)$ by
$\eta_e (T) =
\omega_1 $, if $T$ is not super-ergodic.

\begin{thm} \label{theoreme:rang coa tree ergo}
Let $X$ be a separable Banach space, and $L(X)$ be the space of
bounded operators equipped with the  strong operator topology
$S_{op}$. Let $\eta_e$ be the index on $L(X)$ defined above. Then:

\begin{enumerate}[(a)]
  \item $T$ is super-ergodic if and only if ~  $\eta_e(T) <
\omega_1$.
  \item The set $\mathcal{SE}(X)$ of super-ergodic operators is coanalytic.
  \item $\eta_e$ is a  coanalytic rank on  $\mathcal{SE}(X)$.

\item $\exists \alpha < \omega_1$ such that the set of uniformly ergodic operators $\mathcal{UE}(X)
\subseteq \{T \in \mathcal{SE}(X); \eta_e(T) \leq \alpha\}$.

 \item
$\mathcal{SE}(X)$ is a Borel set
if and only if
{$\displaystyle \eta_e(X):= \sup_{T \in \mathcal{SE}(X)} \eta_e(T) ~
<\omega_1.$}
\end{enumerate}

\end{thm}

\begin{proof}

The assertion $(a)$ is part of the theorem \ref{prop:indice arbre
ergo}.
 Let  $T$ be a bounded operator on $X$ and $A_n$ its
$n^{th}$ Cesaro-mean.
 We can write that
 \[
\eta_e (T) = \sup_{n \in\N}
h\big{(}\mathcal{A}_e(T,\frac{1}{n})\big{)}.
\]
We need to construct  a tree on $\N$ that contains all
trees  $\mathcal{A}_e(T,\frac{1}{n})$  while keeping the
information  on the index $\eta_e(T)$. Let $\mathcal{A}_e(T)$ be
the tree formed by the finite sequences
 $s=(s_0,s_1,s_2,...) $ such that $s_0$
  covers $\N$ and $(s_1,s_2,...)$ covers the trees
$\mathcal{A}_e(T,\frac{1}{s_0})$; i.e., the tree
\[
 \Big{\{}\sigma \in \N^{\uparrow<\N}:
\valu{\sigma}=0 \hbox{ or } \sigma = (k,s) \in \N \times
\N^{\uparrow<\N} \hbox{ with } s \in \mathcal{A}_e(T,\frac{1}{k})
\Big{ \}}.
\]
Consider the map on  the $S_{op}$-Borel set of Cesaro-bounded
$L_{cb}(X)$,
\begin{eqnarray*}
\mathcal{A}_e :  L_{cb}(X) & \longrightarrow &  \{\hbox{Trees on} \N\} \\
                    T & \longmapsto &
\mathcal{A}_e(T).
\end{eqnarray*}

\begin{claim}\label{claim}
 Let $ \sigma \in \N^{<\N} $. Put $\overline{\sigma}=
\{T \in L_{cb}(X); \sigma \in \mathcal{A}_e(T) \}$. Then
$\overline{\sigma}$ is a
 $S_{op}$-Borel subset of $ L_{cb}(X)$.
\end{claim}

 Indeed, this is clear if  the length  of
$\valu{\sigma} \leq 2$ since in this case either
$\overline{\sigma}= L_{cb}(X)$ or $\overline{\sigma}= \emptyset $.
Let
 $\sigma=(k,s) \in \N \times \N^{<\N}$ with $\valu{s} >2 $.
If $s \notin \N^{\uparrow<\N}$ then $\overline{\sigma}=\emptyset$.
If $s \in \N^{\uparrow<\N}$, we have
\begin{align*}
\overline{\sigma}
 &= \Big{\{}T \in L_{cb}(X):
s \in \mathcal{A}_e(T,\frac{1}{k}) \Big{\}}\\
&= \Big{\{}T \in L_{cb}(X): \exists  x \in B_X, \forall 1 \leq p
<\valu{s}; ~ \norm{A_{s_p}x - A_{s_{p+1}}x}> \varepsilon \Big{\}}.
\end{align*}
Since $X$ is separable , let $\{x_n\}_{n \in \N} $ be  dense in
the unit ball   of $X$. Then,
\begin{align*}
\overline{\sigma} &= \Big{\{}T \in L_{cb}(X):  ~  \exists  n \in
\N, ~
  \forall 1 \leq p <\valu{s}; ~
\norm{A_{s_p}x_n - A_{s_{p+1}}x_n}> \varepsilon      \Big{\}}\\
&= \bigcup_{n\in \N}\bigcap_{p=1}^{\valu{s}-1} \Big{\{}T \in
L_{cb}(X):  ~
        \norm{A_{s_p}x_n - A_{s_{p+1}}x_n}> \varepsilon
        \Big{\}}.
\end{align*}
It follows from  the continuity of $T \mapsto T^n$ that
$\overline{\sigma} $ is $S_{op}$-Borel.
\bigskip

Lemma \ref{prop:tree=>rang} below  and  claim \ref{claim} imply
that the set
\[
C := \big{\{} T \in L_{cb}(X); \mathcal{A}_e(T) \hbox{ is well
bounded } \big{\}}
\]
is $S_{op}$-coanalytic in $L_{cb}(X)$ with a coanalytic rank  $h
\circ \mathcal{A}_e$; which maps $T$ to the height of the tree
$\mathcal{A}_e(T)$.
 On the other hand, by the definition of
$\mathcal{A}_e(T)$ and  theorem \ref{prop:indice arbre ergo},
\begin{align*}
C &= \Big{\{} T \in L_{cb}(X); \mathcal{A}(T,\frac{1}{n}) \hbox{
is well founded }
             \forall n \in \N^*\Big{\}}\\
&= \Big{\{} T \in L_{cb}(X); \mathcal{A}(T,\varepsilon ) \hbox{ is
well founded }
             \forall \varepsilon >0\Big{\}}\\
             &= \mathcal{SE}(X) \cap L_{cb}(X)\\
             &= \mathcal{SE}(X).
\end{align*}
Therefore, the set $ \mathcal{SE}(X) $  is
 $S_{op}$-coanalytic in $ L_{cb}(X) $ and thus in $ L(X)$ because
$ L_{cb}(X) $ is a  $S_{op}$-Borel subset of $L(X)$, which proves $(b)$. The index $
\eta_e $ is then a coanalytic rank on $ \mathcal{SE}(X) $ since $
\eta_e = h \circ \mathcal{A}_e $, thus proving $(c)$. The assertions $(d)$ and $(e)$
of the theorem
 follow from  the lemma
\ref{proprietes:rang coa} below on coanalytic ranks.

\end{proof}

Below are adaptations, in particular to the
topology on the set of all trees on $\N$, of classical results of descriptive set
theory used in the proof above   (see \cite{kechrisLouveau} and
\cite{zins}).

\begin{lem} \label{prop:tree=>rang}
Let  $P$ be a Polish space and $\psi$ be a map from  $P$ into the
set of all trees on $\N$.
If for all $s \in \N^{<\N}$, the set $\bar{s}=\{x \in P: s \in
\psi(x)\}$ is Borel, then the set $ C=\{x \in P;  \psi(x) \hbox{
well founded }\} $ is coanalytic  in $P$ with $h\circ \psi$ as
coanalytic rank.
\end{lem}

\begin{lem} \label{proprietes:rang coa}
Let $\delta$ be a coanalytic rank on a
 coanalytic subset $C$ of a Polish space $P$. Then:
\begin{enumerate}[(a)]
  \item $\forall \alpha < \omega_1, ~~ B_\alpha:= \{x \in C;
\delta(x) \leq \alpha\}$ is a Borel set.
  \item If $A \subseteq C$ is analytic, then  $\exists \alpha <
\omega_1$ such that $A \subseteq B_\alpha$.
  \item $C$ is Borel if and only if
$\delta$ is bounded on $C$ by a countable ordinal.
\end{enumerate}

\end{lem}


\bigskip
 The topological hierarchy of these families of
operators makes it possible to put strong limits on possible
relationships among using the index $\eta_e(X)$  of a
space $X$ or the rank $\eta_e(T)$ of an operator $T$.
This generates families of Banach spaces with a desired relationship instead of just individual examples.
 In
particular, if $\eta_e(X)= \omega_1$, then the set of super-ergodic
operators strictly separates  the sets of ergodic and uniformly
ergodic operators. Moreover, not only is an operator $T$ super-ergodic
when its rank $\eta_e(T)$ is countable,   we also have the interesting
dichotomy below. It is an  application of the facts that $\eta_e$ is a
coanalytic rank and the set of ergodic operator is Borel.

\begin{cor}\label{dichotomy}
For every  separable Banach space $X$ exactly one of the following holds:
\begin{itemize}
  \item either there exists an ergodic but not super-ergodic operator on
  $X$,
  \item either there exists  a countable ordinal
$\alpha$
 such that for any ergodic operator $T$ on $X$,  ~
$\eta_e(T) \leqslant \alpha$.
\end{itemize}
\end{cor}

In particular, for any Banach space that contains a complemented
$\ell^1(\N)$, there exists ergodic operators with arbitrary large
ordinals. Indeed, the example 3.3 in \cite{yahdi3} shows that the
left-shift operator $S$ on $ \ell^1(\N)$  is ergodic since $\lim_{n
\to \infty} \norm{S^nx} = 0 $, but it is not super-ergodic since the
$\{A^\mathcal{U}_n(\bar{e})\}_{n \in \N}$ is not Cauchy; where
$A^\mathcal{U}_n$ is the $n^{th}$ Cesaro-mean of $S_\mathcal{U}$ and
$\bar{e}$ is  the classical canonical basis $(e_k)_{k \in \N}$.


\end{document}